\documentclass{amsart}
\usepackage{amsfonts}

\newtheorem{theorem}{Theorem}[section]
\newtheorem{lemma}[theorem]{Lemma}

\theoremstyle{definition}

\newtheorem{proposition}[theorem]{Proposition}
\newtheorem{corollary}[theorem]{Corollary}
\theoremstyle{remark}
\newtheorem{remark}[theorem]{Remark}
\numberwithin{equation}{section}

\begin{document}

\title{Decomposition of Cartan Matrix and Conjectures on  Brauer Character Degrees }

\author{Zeng Jiwen}
\address{School of Mathematics, Xiamen University, Xiamen, 361005, P. R. China}
\curraddr{School of Mathematics, Xiamen University, Xiamen, 361005,
P. R. China} \email{jwzeng@xmu.edu.cn}
\thanks{}




\dedicatory{School of Mathematics, Xiamen University, Xiamen,
361005, P. R. China.}


\begin{abstract}
Let $G$ be a finite group and $N$ be a normal subgroup of $G$.  Let
$J=J(F[N])$ denote the Jacboson radical of $F[N]$ and $I={\rm
Ann}(J)=\{\alpha \in F[G]|J\alpha =0\}$. We have another algebra
$F[G]/I$. We study the decomposition of Cartan matrix of $F[G]$
according to $F[G/N]$ and $F[G]/I$. This decomposition establishs
some connections between Cartan invariants and chief composition
factors of $G$. We find that existing zero-defect $p$-block in $N$
depends on the properties of $I$ in $F[G]$ or Cartan invariants.
When we consider the Cartan invariants for a block algebra $B$ of
$G$, the decomposition is related to what kind of blocks in $N$
covered by $B$. We mainly consider a block $B$ of $G$ which covers a
block $b$ of $N$ with $l(b)=1$. In two cases, we prove Willems'
conjecture holds for these blocks, which covers some true cases by
Holm and Willems. Furthermore We give an affirmative answer to a
question by Holm and Willems in our cases. Some other results about
Cartan invariants are presented in our paper.
\end{abstract}

\maketitle


%

\section{Introduction}
Let $A$ be a Frobenius algebra and $J(A)$ be the Jacboson radical,
then the relations between the Cartan matrix of $A$ and that of
$A/J(A)^{i}, i\geq 1$, is studied in \cite{some}. Now suppose
$A=F[G]$ is a group algebra for a finite group $G$, if $N$ is a
normal subgroup of $ G$, we have known some results about relations
between $F[G]$ and $F[G/N]$. For example, from the Alperin, Collins
and Sibley\cite{acs}, we can see some connections between Cartan
matrix of $F[G]$ and that of $F[G/N]$. A result of
Lusztig\cite[P162, Lemma 4.26]{feit} also implies their connections.
For a number of relations between projective modules of $F[G]$ and
$F[G/N]$, see W. Willems\cite{willem}.

When $N$ is a normal subgroup of $G$, let $I$ be the annihilator of
the Jacobson radical $J(F[N])$ in $F[G]$. Then $I$ will be an ideal
of $F[G]$. In this paper we shall establish some relations among
$F[G], F[G/N]$, and $F[G]/I$. One important result is the
decomposition of Cartan matrix of $G$ in terms of Cartan matrices of
$G/N$ and $F[G]/I$. As an application of this fact, we shall give
connections between Cartan invariants and chief composition factors
of $G$. We generalize similar results in \cite{tsu}.

When $B$ is a $p$-block algebra of $G$, we find that the
decomposition of Cartan matrix of $B$ is heavily related to what
 kind of $p$-blocks of $N$ covered by $B$. There are two cases that
we have more advantage to study. One is a zero-defect block in $N$.
Another is a block of $N$ with only one irreducible Brauer
character. So we shall discuss in what conditions $B$ will cover a
zero-defect $p$-block algebra of $N$. Our discussions produce some
conditions of existing zero-defect $p$-block algebra for $N$ in
terms of $I$ or Cartan invariants.

For a block $B$ of $G$, let $k(B)$ and $l(B)$ denote the numbers of
ordinary irreducible characters and irreducible Brauer characters of
$B$, respectively. We mainly discuss the Cartan invariants of $B$ if
it covers a block $b$ in $N$ of $l(b)=1$.

W. Willems and T. Holm \cite{willems}\cite{holm} present several
conjectures on Brauer character degrees. Although it is not easy to
prove their conjectures in general, they show that  there are
several affirmative answers for these conjectures. In our cases,
which cover some true cases proved by them, their conjectures are
proved true in this paper.

Let $p$-block $B$ of $G$ has defect group $D$ and Cartan matrix $C$,
W. Willem and T. Holm present a question: is it true that
$$\textrm{Tr}(C)\leq l(B)|D|?$$
We give affirmative answer for their question in our cases.

As application for our results, we also consider the bounds of
Cartan invariants in terms of defect groups. The motivity comes from
Brauer Problem VII\cite[Ch. IV, \S5]{feit}. We prove the Cartan
invariants are bounded by the order of defect group in our cases.

An interesting work of Cartan matrix is studying its eigenvalue, see
Wada's paper\cite{wada}. Here we present a similar result which
generalize a result of \cite{wada}\cite[Ch. IV, Lemma 4.26]{feit}.

Cartan matrix plays an key role in modular representation of finite
groups, so we can find a lot of articles about Cartan matrix,
see\cite{kul} \cite{rob} \cite{shi} .

This paper is organized as follows. Section 2 introduces two basic
Lemmas; In section 3 we discuss relations among $F[G], F[N]$ and
$F[G]/I$; Section 4 provides results about connections between block
algebras of $F[G]$ and $F[G]/I$; Section 5 gives a decomposition of
Cartan matrix of $G$ in terms of Cartan matrices of $F[G/N]$ and
$F[G]/I$; In section 6 we study Willems' conjecture on Brauer
character degrees. In our cases we prove Willems' conjecture and
give an affirmative answer to a Willems' question. Section 7 is the
application of previous sections. We present some results on Cartan
invariants.

 In this paper, a $G$-module usually means a
left $G$-module
 unless we give it under specified conditions. $F$ is always a splitting field of character $p\neq 0$.
 For any two $G$-modules $M, P$,
  we use $(M, P)^{G}$ denote ${\rm
Hom}_{F[G]}(M, P)$ in this article. We fix $A=F[G], J=J(F[N])$, the
Jacobson radical of $F[N]$, and $I=\{\alpha \in F[G]|J\alpha =0\}$.
Some notations and basic
 results are refered to \cite{some}.

\section{ Lemmas}

Let $N \unlhd G$ and  $J=J(F[N])$ denote the Jacboson Radical of
$F[N]$. We consider $J(F[N])$ and $F[N]$ as $G$-module by
$G$-conjugate action on them. For any $g \in G$, since $g J(F[N])=
J(F[N])g$, so $F[G] J(F[N])$ is a nilpotent ideal of $F[G]$, or
$F[G] J(F[N])\leq J(F[G])$. If $S$ is a subset of $F[G]$, let
$$r(S)={\rm Ann}_{r}(S)=\{a \in F[G] | Sa=0\}$$ denote the right
annihilater ideal decided by  $S$.

\begin{lemma} Let $J=J(F[N])$ and let $r(J)$ denote the right annihilater
ideal of $J$ in $F[G]$, then  $r(J)=r(F[G] J(F[N]))$. Thus $r(J)$ is
an ideal of $F[G]$.
\end{lemma}
\begin{proof} $r(F[G] J(F[N]))\leq r(J)$ is easy to know. If $a
\in r(J)$, then $Ja=0$. Thus $F[G] J(F[N])a=0$, which induces $a \in
r(F[G] J(F[N]))$.
\end{proof}
\begin{remark}
 In the following sections, we use $I$ denote $r(J)$. According
to Y. Tsushima\cite{tsu}, we can find an element $c$ in the center
of the group algebra $F[G]$ such that $I=F[G]c$. Hence
$F[G]/F[G]J(F[N])$ is a symmetric algebra too.
\end{remark}
 Lemma 2.1 makes the following discussions reasonable.

If $M$ is a right $G$-module, then we let $M^{*}={\rm Hom}_{F}(M,
F)$ denotes the dual module under usual way, so it is a left
$G$-module. Now we let $\{P_{1}, P_{2}, \cdots, P_{n}\}$ denote a
complete set of representatives of the isomorphism classes of
principal indecomposable right $G$-modules. Since $F[G]$ is a
symmetric algebra, then $\{P^{*}_{1}, P^{*}_{2}, \cdots,
P^{*}_{n}\}$ is a complete set of representatives of the isomorphism
classes of principal indecomposable left $G$-modules when we
consider $F[G]$ as a left regular $G$-module. We recall that the
Cartan matrix $(c_{ij})$ decided by $\{P_{1}, P_{2}, \cdots,
P_{n}\}$ is $c_{ij}=$ the multiplicity of Hd$(P_{j})$ as a
composition factor in $P_{i}$\cite{some}.

\begin{lemma} Let $S(M^{*})=\{f \in M^{*}| J(F[N])f=0\}$, then $S (M^{*})=(M/MJ)^{*}$, where $J=J(F[N])$.
\end{lemma}
\begin{proof} We know that $S(M^{*})$ is a $G$-submodule of
$M^{*}$. If $f \in (M/MJ)^{*}={\rm Hom}_{F}(M/MJ, F)$, then
$f(MJ)=0$, which means $Jf(M)=0$, so $f \in S(M^{*})$. Conversely,
if $f \in S(M^{*})$, then $f(MJ)=0$ and $MJ \leq {\rm Ker}(f)$,
which induces $f \in (M/MJ)^{*}$.  \end{proof}

It is easy to know that Soc$(M^{*}) \leq S(M^{*})$.
\begin{remark}
 As usual way, we can define $U^{\bot}=\{f\in M^{*}|f(u)=0,u\in
U\}$, for $U\leq M$. Then we can prove that $S(M^{*})=(MJ)^{\bot}$,
so naturally we have $(MJ)^{\bot}=(M/MJ)^{*}$.
\end{remark}

\section{ Relations among algebra $F[G], F[N]$ and $F[G]/I$}
Let $A=F[G]$ and $I=r(J)$ as before. In this section, we study the
the structure of the algebra $A/I$, and discuss the relations among
$F[G], F[N]$, and $A/I$.

About the algebra $F[G]/I$, we have the following results to
describe it.
\begin{proposition}Let $A=F[G], N\unlhd G, J=J(F[N]),
I=r(J)=\{\alpha\in F[G]|J\alpha=0\}$,
  and as left regular module, let $A=\bigoplus_{P} P$  be a
  decomposition of principal indecomposable modules of $A$ ,
then $A/I=\bigoplus_{P} P/P\cap I$,  where $P/P\cap I$   is
principal indecomposable $A/I$-module if $P/P\cap I\neq 0$.
\end{proposition}
\begin{proof} Let $\alpha \in I, \alpha=\sum_{P}\alpha_{P}, \alpha_{P}\in
P$, Since $J\alpha=0$, so $J\alpha_{P}=0, \alpha_{P}\in P\cap I$.
Thus $I=\bigoplus_{P}P\cap I$, which implies
$A/I=\bigoplus_{P}P/P\cap I$.

In order to show that $P/P\cap I$ is indecomposable, we consider
Hd($P/P\cap I$) and prove it is a simple module. First we claim that
$J(A/I)=(J(A)+I)/I$. This is because $A/(J(A)+I)$ is semi-simple,
and if $A/M$ is simple, where $I\leq M$, then $J(A)+I \leq M$. Thus
we obtain $J(P/P\cap I)=(J(A)P+P\cap I)/P\cap I$, which shows
Hd$(P/P\cap I) = P/(J(A)P+P\cap I)$. Since $P/(J(A)P+P\cap I)$ is a
factor module of $P/J(A)P$, we have $J(A)P+P\cap I=J(A)P$ or
$J(A)P+P\cap I=P$. So $P\cap I \leq J(A)P$ or $P\cap I =P$ by
Nakayama lemma. Hence we have Hd$(P/P\cap I)=$ Hd$(P/J(A)P)$ if
$P\cap I \neq P$.
\end{proof}
\begin{remark}
 There is a primitive idempotent $e$ such that $P=Ae$ when $P$ is
left indecomposable ideal of $A$  , so we can write $P\cap I=IP=Ie$
and $P/P\cap I = Pe/Ie$.\end{remark}

 The proof of the result above induces the following result
which suggests the relations between projective modules of $F[G]$
and $F[G]/I$.
\begin{corollary} With the same notations as above, let $E=$
Hd$(P/J(A)P)$ be a simple $A$-module. Then $E$ is still a simple
$A/I$ module if and only if $P \cap I \neq P$. Furthermore, $E=$
Hd$(P/P\cap I)$. $\Box$
\end{corollary}
\begin{remark}
 From this result, for a simple $F[G]/I$-module $E$, if $P$ is
the projective $F[G]$-cover of $E$, then $P/P\cap I$ is the
projective $F[G]/I$-cover of $E$. A very natural question arises:
What happens with $E$ being not simple $A/I$-module when $E$ is a
simple $G$-module? We state the following result to answer this
question.
\end{remark}
\begin{proposition} With the same notations as above, let $E={\rm
Hd}(P/J(A)P)$. Suppose $E$ lie over a simple $N$-module $U$. Then
$P\leq I$ if and only if $U$ is a  projective simple $N$-module.
\end{proposition}
\begin{proof} If $P\leq I$, then $J(F[N])P=0$, so $P_{\downarrow N}$ is
a direct sum of projective simple $N$-modules. Since $E_{\downarrow
N}={\rm Soc}(P)_{\downarrow N}\leq P_{\downarrow N}$, $E_{\downarrow
N}|P_{\downarrow N}$, so $U|P_{\downarrow N}$ and $U$ is projective.

Conversely, let $T$ is the inertial group of $U$ in $G$. By Clifford
theory, we have $E=F[G]\otimes_{T} \overline{U}$, where
$\overline{U}=\sum_{t\in T}tU$. We can write $
\overline{U}=\sum_{t\in T}tU=\sum_{t\in T/N}tU$ and $F[T]\otimes_{N}
U=\oplus_{t\in T/N}t\otimes U$, then we set a map $f:F[T]\otimes_{N}
U \longrightarrow \overline{U}$ by $\sum_{t\in T/N}t\otimes
u_{t}\mapsto \sum_{t\in T/N}tu_{t}$, which is a surjective
$T$-homomorphism. Thus we get a surjective $G$-homomorphism
$$1\otimes f: F[G]\otimes U\longrightarrow F[G]\otimes \overline{U}=E$$
Since $U$ is projective $N$-module and $P$ is the projective cover
of $E$, then $P|F[G]\otimes U =U^{\uparrow G}$. $(U^{\uparrow
G})_{\downarrow N}=\bigoplus_{g\in G/N}g\otimes U$ is the sum of
projective simple $N$-modules, so is $P_{\downarrow N}$. Thus
$J(F[N])P=0, P\leq I$.

\end{proof}

The result above tells us an important fact is : When $P$ is a
principal indecomposable module of $G$, then $P_{\downarrow N}$ is
semi-simple if and only if Hd$(P)$ lie over a projective simple
$N$-module.

According to Proposition 3.1, there exists  a principal
indecomposable $G$-module $P$ such that $P\leq I$ if and only if the
number of classes of isomorphic simple $A/I$-modules is smaller than
that of isomorphic simple $G$-modules, which is the dimension ${\rm
Dim}_{F}Z(A/J(A))$ of the center $Z(A/J(A))$ of $A/J(A)$ when $F$ is
the splitting field of $A$. Hence we give the following two results,
which offer  conditions for  existence of $p$-block of defect zero
in $F[N]$.
\begin{corollary} Suppose $A=F[G]$ and $F$ is the splitting field of $A$.
Let $N$ be a normal subgroup of $G$, $J=J(F[N])$, and $I=\{\alpha\in
F[G]|J\alpha =0\}$. Then there not exist projective simple
$N$-modules if and only if $I\leq J(A)$.
\end{corollary}
\begin{proof} Notice that $J(A/I)=(I+J(A))/I$ and let $\overline{A}=A/I$, then $${\rm Dim
}_{F}Z(A/(I+J(A)))={\rm
Dim}_{F}Z(\overline{A}/J(\overline{A})),$$which is the number of
classes of isomorphic simple $A/I$-modules.  Since $${\rm
Dim}_{F}Z(A/J(A))=l(G)$$ when $A=F[G]$, the assertion holds.
\end{proof}

\section{relations between blocks of $G$ and $F[G]/I$ }
 Now we are going to
prove some results for  block algebras of $G$ and $F[G]/I$.
\begin{theorem} Let $P_{i}, P_{j}$ be any two principal
indecomposable modules belonging to the same $p$-block of $G$. Then
$P_{i}\leq I$ if and only if $P_{j}\leq I$.
\end{theorem}
\begin{proof} Since any two principal indecomposable modules are
linked \cite{lan} if they are in the same block, we only need to
suppose ${\rm Hom}_{F[G]}(P_{i},P_{j})=(P_{i},P_{j})^{G}\neq 0$.

Let $E_{i}={\rm Hd}(P_{i}), E_{j}={\rm Hd}(P_{j})$ and  they are
supposed to lie over simple $N$-modules $U_{i}$ and $U_{j}$,
respectively. By Proposition 3.5, it is sufficient to prove: if
$U_{i}$ is projective, then $U_{j}$ is projective too.

$0\neq(P_{i}, P_{j})^{G}\leq (P_{i}, P_{j})^{N}$ yields $(P_{i},
P_{j})^{N}\neq 0$. In the proof of Proposition 3.5, we have proved
$P_{i}|(U_{i})^{\uparrow G}$, and consequently $(P_{i})_{\downarrow
N}=\oplus(g\otimes U_{i})$, a direct sum of projective simple
$N$-modules. Then $$(P_{i}, P_{j})^{N}=(\oplus(g\otimes U_{i}), {\rm
Soc}((P_{j})_{\downarrow N}))^{ N}.$$ Next we claim that:

\emph{for a simple $N$-module $U$,  $U|(E_{j})_{\downarrow N}$ if
and only if $U | {\rm Soc}((P_{j})_{\downarrow N}).$}

As $E_{j}={\rm Soc}(P_{j})\leq P_{j}$, $(E_{j})_{\downarrow N}\leq
{\rm Soc}((P_{j})_{\downarrow N})$. If $E$ is a simple $G$-module
lie over a simple $N$-module $U$ such that $U|{\rm
Soc}((P_{j})_{\downarrow N})$, then $E=\sum_{g\in G} gU \leq P_{j}$
and so $E=E_{j}$. Hence $U|(E_{j})_{\downarrow N}$.

Hence ${\rm Soc}(P_{j})_{N}$ and ${\rm Soc}\left((P_{j})_{N}\right)$
are decomposed a sum of simple $N$-modules which are conjugate to
the same simple $N$-module. Finally, since
$$0\neq (P_{i},P_{j})^{N}=(\oplus(g\otimes U_{i}), {\rm
Soc}((P_{j})_{\downarrow N}))^{ N}$$ so$$(\oplus(g\otimes U_{i}),
(E_{j})_{\downarrow N})^{ N}\neq 0$$ by the argument we claim above.
Therefore  there exist some $g, x\in G, g\otimes U_{i}\cong x\otimes
U_{j}$. So $U_{j}$ is projective as $U_{i}$ is projective.
\end{proof}
\begin{remark}
 By the result above, if a principal indecomposable module $P$ is
in a $p$-block algebra $B$, then $P\leq I$ if and only if $B\leq I$.
\end{remark}

We consider relations between $I$ and a block $B$ of $G$, then the
previous results Proposition 3.1 and Theorem 4.1 tell us that $B\leq
I$ or $B\cap I =Ie_{B}\leq J(B)=BJ(F[G]) =J(F[G])e_{B}$, where
$e_{B}$ denote the block idempotent of $B$. For a block $b$ of $N$,
let Bl$(G|b)$ denote all blocks of $G$ which lie over $b$. Let
Bl$(N|0)$ denote all blocks of $N$ of defect zero.  Thus we have the
following result to describe $I$.

\begin{corollary}Let $N\unlhd G$ and $N_{p}$ denote the set of all
$p$-elements of $N$. Let $c=\sum_{x\in N_{p}}x$. Then
$$I=F[G]c^{2}\bigoplus( \bigoplus_{e_{B}\not\in I}Ie_{B}).$$
In particular, $I\leq J(F[G])$ if and only if $c^{2}=0$.
\end{corollary}

\begin{proof}By the previous results, we can write $$
I=(\bigoplus_{e_{B}\in I}F[G]e_{B} )\bigoplus
(\bigoplus_{e_{B}\not\in I}Ie_{B}),$$so we just need to prove
$\bigoplus_{e_{B}\in I}F[G]e_{B} =F[G]c^{2}$.

Notice $e_{B}\in I$ if and only if $B$  lie over a block $b$ of $N$
with defect zero. By Tsushima\cite{tsus}, $c^{2}$ is the sum of all
block idempotents $e_{b}$ in $F[N]$ with defect zero. Hence we have
$$\begin{array}{lll}c^{2} & = & \sum_{b\in
\textrm{Bl}(N|0)}e_{b}\\ & = &
\sum_{b_{i}}\sum_{b_{i}^{t}}e_{b_{i}}^{t}\\
&  & (\textrm{where $b_{i}$ denote the representive of $G$-orbit of Bl($N|0$)})\\
& = &\sum_{b_{i}}\textrm{Tr}_{T(b_{i})}^{G}(e_{b_{i}}) \\
&  & (\textrm{sum of some central idempotents in $G$} )\\
& = & \sum_{b_{i}}\sum_{B \in \textrm{Bl}(G|b_{i})}e_{B}\\& = &
\sum_{B \leq I}e_{B}
\end{array}$$
which implies our result.

\end{proof}
 If $N$ is a $p$-solvable group, then we can prove $I=F[G]c$ from
 result in \cite[Chp $X$,P422]{feit}. So $I\leq J(F[G])$ if and only
 if $c\in J(F[G])$. If $N$ is not $p$-solvable group, the Corollary 4.3
 still implies $I\leq J(F[G])$ if and only
 if $c\in J(F[G])$, as $c^{2}$ is an idempotent. Thus Corollary 4.3
 is conformed with  Corollary 3.6.

An immediate consequence of the results above is
\begin{theorem} Let $F[G]=A=\bigoplus_{i}B_{i}$ be a
decomposition of $p$-block algebra of $A$. $\overline{B_{i}}$
denotes the image of $B_{i}$ under the natural map:$A\rightarrow
A/I$. Then
\begin{enumerate}
\item $\overline{B_{i}}=0$ if and only if $(B_{i})_{\downarrow N}$
is the sum of $p$-blocks of defect zero in $F[N]$. \item if
$\overline{B_{i}}\neq 0$, then  the  simple $G$-modules of $B_{i}$
are also the simple $A/I$-modules of $\overline{B_{i}}$ \item $A/I =
\bigoplus_{i}\overline{B_{i}}$, where $\overline{B_{i}}=0$ or
$B_{i}/B_{i}\cap I$ with $B_{i}\cap I\leq J(B_{i})$.       $\Box$
\end{enumerate}
\end{theorem}

\section{ An algebraic decomposition of Cartan matrix}

In the following, we are going to give a relation among the Cartan
matrices of $F[G]$, $F[G/N]$ and $F[G]/r(J(F[N]))$. Now let
$C_{G}=(c_{ij})$ denote the Cartan matrix of $G$ and
$C_{\overline{G}}=(\overline{c_{ij}})$ denote the Cartan matrix of
$F[G/N]$. For simple $F[G/N]$-module $E_{i}$, let $\overline{P_{i}}$
 and $P_{i}$ denote the projective $F[G/N]$--cover and $F[G]$-cover of $E$, respectively. Then by
\cite[Chapter 2, \S 11]{lan}, we have $\overline{P_{i}}\cong
P_{i}/P_{i}J$ as $F[G]$-module, where $J=J(F[N])$.

 To understand our proof in the following result, we need to
know the basic fact: as Frobenius algebra, $_{A}A\cong (A_{A})^{*}$,
so if $\{P\}$ is right principle indecomposable module, then
$\{P^{*}\}$ is left principle indecomposable module.

 In order to make our discussion more clearly, We divide the set of
 all right principal indecomposable modules(up to isomorphisms) of the group $G$ into two
 subsets: let $$S_{1}=\{P \;|\;{\rm Hd}(P)\; {\rm a \;simple\; module \;of} \;G/N\}$$ and
 $$S_{2}=\{P\;|\;
 {\rm Hd}(P) {\rm \; a\; simple \; module \;} \not\in G/N\}.$$ We
  describe the difference between $S_{1}$ and $S_{2}$ as follows:

\begin{proposition} Let $J=J(F[N])$ be the Jacoboson radical of $F[N]$ and
$A_{N}$ the augmentation ideal of $F[N]$.  1) If $P \in S_{1}$, then
$PJ=PA_{N}$; 2) If $P \in S_{2}$,  then  $PA_{N}=P$.
\end{proposition}
\begin{proof}
\begin{enumerate}
\item  In fact it was proved in \cite[Chapter 2, \S 11]{lan}.

\item Let $E={\rm Hd}(P)$ for $P\in S_{2}$, then $E \not\in G/N$.
If $PA_{N}\neq P$, then $PA_{N}\leq PJ(F[G])$. Let
$\bar{P}=P/PA_{N}$, then $\bar{P}$ is a $G/N$-module which has
Hd($P$) in $G/N$, a contradiction.
\end{enumerate}
\end{proof}

We let the Cartan matrix $C=(c_{ij})$  of $F[G]$ be decided by the
set of right principal indecomposable modules:$P_{1}, P_{2}, \ldots,
P_{l}$. Then
$$ \begin{array}{lll} c_{ij}& = &{\rm
Dim}_{F}(P_{j}, P_{i})^{G}\\ & = &{\rm Dim}_{F}(P_{i},
P_{j})^{G}\\ & = & {\rm Dim}_{F}(P_{j}^{*}, P_{i}^{*})^{G}\\
& = & {\rm Dim}_{F}(P_{i}^{*}, P_{j}^{*})^{G}, \end{array}$$ that is
to say $c_{ij}=c_{ji}=c_{ij}^{*}=c_{ji}^{*}$\cite{some}. Let
$(\overline{c_{ij}})$ be the corresponding Cartan matrix of
$F[G]/F[G]J$ decided by $\{P_{i}/P_{i}J\}$, where $J=J(F[N])$. Then
$\overline{c_{ij}}={\rm Dim}_{F}(P_{j}/P_{j}J, P_{i}/P_{i}J)^{G}$.
Since $F[G]/F[G]J$ is also a symmetric algebra, so
$\overline{c_{ij}}=\overline{c_{ji}}=\overline{c_{ij}}^{*}=\overline{c_{ji}}^{*}$.
The following result helps us to calculate $\overline{c_{ij}}^{*}$.

\begin{lemma} With the same notations as above, $N$ is a normal subgroup of $G$. Let $P$ is the right
principal indecomposable module with head ${\rm Hd}(P)$ in
$F[G]/F[G]J$, where $J=J(F[N])$. Then $(P/PJ)^{*}=P^{*}/JP^{*}$.
\end{lemma}
\begin{proof} We should notice both $(P/PJ)^{*}$ and
$P^{*}/JP^{*}$ are left principal indecomposable modules of
$F[G]/F[G]J$. If we prove they have the same head, they are equal.

First we have $({\rm Hd}(M))^{*}={\rm Soc}(M^{*}), {\rm
Hd}(M^{*})=({\rm Soc}(M))^{*}$ for any $G$-module $M$ by duality.
Hence ${\rm Hd}((P/PJ)^{*})=({\rm Soc}(P/PJ))^{*}=({\rm
Hd}(P/PJ))^{*}=({\rm Hd}(P))^{*}$. On the other hand, we have ${\rm
Hd}(P^{*}/JP^{*})={\rm Hd}(P^{*})=({\rm Soc}(P))^{*}=({\rm
Hd}(P))^{*}$. So they have the same head.
\end{proof}

\begin{theorem}With the same notations as above, let $J=J(F[N])$ and
$I=r(J(F[N]))$. We use $\overline{C}=(\overline{c_{ij}})$ and
$C_{J}=(x_{ij})$
 to denote the Cartan matrix of $F[G/N]$ and $F[G]/I$, respectively. Then
 \begin{enumerate}
 \item If $P_{i}, P_{j} \in S_{1}$, then
 $c_{ij} = \overline{c_{ij}} + x_{ij}$
 \item If $P_{j}\in S_{1}$ and $P_{i}\in S_{2}$, then
 $c_{ij}=x_{ij}$
 \item If $P_{i}, P_{j} \in S_{2}$, then $c_{ij}=a_{ij}+x_{ij}$,
 where $a_{ij}$ is the Cartan number of Algebra $F[G]/F[G]J$.

\end{enumerate}
\end{theorem}
\begin{proof}
\begin{description}

\item[Case One] For projective module $P_{j}$,  we have
\begin{equation}
0\longrightarrow P_{j}J \longrightarrow P_{j} \longrightarrow
 P_{j}/P_{j}J\longrightarrow 0,
 \end{equation}
 so we have $$
0\longrightarrow  (P_{i}, P_{j}J)^{G}\longrightarrow (P_{i},
P_{j})^{G} \longrightarrow (P_{i}, P_{j}/ P_{j}J)^{G}\longrightarrow
0$$ as $P_{i}$ is $G$-projective module. Thus $$ {\rm
Dim}_{F}(P_{i}, P_{j})^{G} = {\rm Dim}_{F}(P_{i}, P_{j}J)^{G} +
 {\rm Dim}_{F}(P_{i}, P_{j}/ P_{j}J)^{G}.$$
Since $\overline{P_{i}}\cong P_{i}/P_{i}J$, so we got
 $$c_{ij}={\rm Dim}_{F}(P_{i}, P_{j}J)^{G} + \overline{c_{ij}}.$$

 Now for ${\rm Dim}_{F}(P_{i}, P_{j}J)^{G}$, first we know by \cite{some}
 \begin{equation}
 {\rm Dim}_{F}(P_{i}, P_{j}J)^{G}  = {\rm Dim}_{F}((P_{j}J)^{*}, P_{i}^{*})^{G}.
  \end{equation}
 On the other hand, we have the
following from (4.1)
\begin{equation} 0 \longrightarrow
(P_{j}/P_{j}J)^{*} \longrightarrow P_{j}^{*} \longrightarrow
(P_{j}J)^{*} \longrightarrow 0.
 \end{equation}
 By
Lemma 2.3, we have $S(P_{j}^{*}) = ( P_{j}/P_{j}J)^{*}$, then by
(4.3) we get
$$(P_{j}J)^{*} \cong P_{j}^{*}/S(P_{j}^{*}).$$
 Thus $$\begin{array}{lll}{\rm dim}_{F}((P_{j}J)^{*}, P_{i}^{*})^{G}&=&
 {\rm dim}_{F}(P_{j}^{*}/S(P_{j}^{*}), P_{i}^{*})^{G}\\
& =&{\rm Multiplicity\; of\; Soc}(P_{i}^{*})\;{\rm as\; the\; composition\; factor\; in\; }\\&  &P_{j}^{*}/S(P_{j}^{*})\\
 & =&{\rm Multiplicity \;of \;Hd}(P_{i}^{*}){\rm \;as \;the \;composition \;factor \;in \;}\\&   &P_{j}^{*}/S(P_{j}^{*})\\
 &=& x_{ji}^{*}\end{array}$$
 Notice $P_{j}^{*}\cap
I=S(P_{j}^{*})$, then by (4.2), we have
\begin{equation} c_{ij} =
\overline{c_{ij}} + x_{ji}^{*},
\end{equation} where
$(x_{ji}^{*})$ is the left Cartan matrix of $F[G]/r(J)$, decided by
$$\{P_{1}^{*}/S(P_{1}^{*}), P_{2}^{*}/S(P_{2}^{*}), \cdots,
P_{n}^{*}/S(P_{n}^{*})\}.$$

Similarly we can get
$$\begin{array}{lll}c_{ij}^{*}& = &{\rm
Dim}_{F}(P_{i}^{*}, JP_{j}^{*})^{G}+{\rm Dim}_{F}(P_{i}^{*},
P_{j}^{*}/JP_{j}^{*})^{G}\\ & = &{\rm Dim}_{F}((JP_{j}^{*})^{*},
P_{i})^{G}+{\rm Dim}_{F}(P_{i}^{*}/JP_{i}^{*},
P_{j}^{*}/JP_{j}^{*})^{G}\\ & = & {\rm Dim}_{F}((JP_{j}^{*})^{*},
P_{i})^{G}+\overline{c_{ij}}^{*}
\end{array}$$
by  Lemma 4.2. Since $(JP^{*})^{*} =P/I\cap P$ by the same arguments
as above, $$\begin{array}{lll}{\rm Dim}_{F}((JP_{j}^{*})^{*},
P_{i})^{G} &= &{\rm Dim}_{F}(P_{j}/I\cap P_{j}, P_{i})^{G}\\&= &
{\rm the \;number\; of \;Soc} (P_{i})\\ &  &{\rm as\; the
\;composition \;factors
\;of} P_{j}/I\cap P_{j}\\& =& {\rm the \;number \;of \; Hd }(P_{i})\\
& &{\rm as \;the \;composition\; factors\; of} P_{j}/I\cap P_{j}\\&
= & {\rm the \;number \;of  \;Hd }(P_{i}/I\cap P_{i})\\ &
&{\rm as \;the \;composition \;factors\; of } P_{j}/I\cap P_{j} \\
&= &x_{ji}.
\end{array}$$ Therefore we have
\begin{equation}c_{ij}^{*}= \overline{c_{ij}}^{*} + x_{ji}.
\end{equation}
 Since the Cartan matrices of $F[G]$ and
$F[G/N]$ are dual and symmetric by \cite{some}, so  $c_{ij} =
\overline{c_{ij}} + x_{ij}$ by equality (4.4) and (4.5).

  \item[Case two] Let $P_{i} \in S_{2}, P_{j}\in S_{1}$. Then
  $$\begin{array}{lll}(P_{i}, P_{j})^{G}& = & (P_{i}A_{N},
  P_{j})^{G}\\ &= &(P_{i}, P_{j}J)^{G} \\
\end{array}$$
By the same arguments as in Case one, we have $c_{ij}=x_{ji}^{*}$.
Thus $c_{ij}=x_{ij}$ as $c_{ij}$ is dual and symmetric.

\item[Case three] Let $P_{i}, P_{j} \in S_{2}$. Then by the same
reason as in Case one, we have $$ {\rm Dim}_{F}(P_{i}, P_{j})^{G} =
{\rm Dim}_{F}(P_{i}, P_{j}J)^{G} +
 {\rm Dim}_{F}(P_{i}, P_{j}/ P_{j}J)^{G}.$$

 Let $a_{ij}={\rm Dim}_{F}(P_{i}, P_{j}/ P_{j}J)^{G}={\rm Dim}_{F}(P_{i}/P_{i}J, P_{j}/ P_{j}J)^{G}$, which is
 Cartan number from the symmetric algebra $F[G]/F[G]J$, so $a_{ij}$ is dual and
 symmetric. The same arguments as before induces
 $c_{ij}=a_{ij}+x_{ij}$.

  Then the assertion follows.
  \end{description}
   \end{proof}

  Consider two extreme examples for Theorem 4.3. If $N$ is a
  $p'$-prime group, then all $x_{ij}=0.$ The Cartan numbers from $S_{2}$ tell nothing new by Theorem 4.3.
  If $N$ is a $p$-group,
  then set $S_{2}$ is an empty set. Only Case one could happen.

  We apply the results above to consider the Cartan matrix of a
  $p-$block algebra of $G$. If $B$ is a $p$-block algebra of $G$
  covering a $p$-block algebra of $N$ of defect zero, then $B$ has
  no simple modules in $F[G]/I$ by Theorem 3.10. Hence all Cartan numbers $x_{ij}$
  of $F[G]/I$
  related to the principal indecomposable modules in $B$ will be
  zero. Furthermore, we know the Cartan matrix of a block algebra
  can not split into two unconnected parts, which shows that all principal indecomposable modules of $B$
  are in $S_{1}$ or all in $S_{2}$. So we have the
  following
  \begin{corollary} Let $N$ be a normal subgroup of group $G$. Then $B$ is a $p$-block algebra of  $G$
  covering a zero-defect $p$-block of $N$ if and only if the Cartan matrix of
  $B$ is equal to that of image of $B$ in $F[G/N]$ or that of image
  of $B$ in $F[G]/F[G]J$, where $J=J(F[N])$.
  \end{corollary}
  \begin{proof} The assumption must holds by the arguments above if $B$ covers a zero-defect $p$-block of $N$.
  Conversely, if the Cartan matrix of $B$ is equal to that of
  image of $B$ in $F[G/N]$ or that of image
  of $B$ in $F[G]/F[G]J$, we can find some principal
  indecomposable module $P_{i}$ in $B$ such that Cartan number $x_{ii}=0$
  of $F[G]/I$ corresponding to $P_{i}$. Thus $P_{i}\leq I$ and $B$
  covers a zero-defect $p$-block of $N$ from Theorem 3.10.
  \end{proof}

\begin{remark}
 If $N$ is $p'$-group and the image of $B$ in $F[G/N]$ is not
zero, it is well known that $B$ and its image have the same
irreducible characters. We generalize these results.
\end{remark}
Notation: In the following, when we say a Cartan number $c$ is
corresponding to a simple module $E$, it means that $c$ is the
multiplicity of $E$ as a composition factors in the projective cover
of $E$.

\begin{corollary}Let $N$ be a normal subgroup of $G$ and $G/N$ is
a $p$-group. Number $c_{11}$ denote the Cartan number corresponding
to the trivial $G$-module $F$. Then $c_{11}\geq|G/N|$ and the
equality holds if and only if $G$ is $p$-nilpotent and
$N=O_{p'}(G)$.
\end{corollary}
\begin{proof} Since $G/N$ is a $p$-group, the set $S_{1}$ contains
only one element $P$ which is the projective cover of trivial
$G$-module $F$. The Cartan number $\overline{c_{11}}$ in $G/N$
corresponding to the trivial $G/N$-module is $|G/N|$, so $c_{11}\geq
|G/N|$. The equality holds if and only if the principal $p$-block
covers a defect-zero block in $N$, but this is equal to the
principal $p$-block having only one irreducible Brauer character.
\end{proof}

In the following we discuss the relations between the composition
factors and Cartan numbers for a group $G$. We will generalize a
result in \cite{tsu} about a lower bound for the first Cartan
invariant $c_{11}$ in terms of the chief composition factors of $G$.

Let the group $G$ has a series of normal subgroups as following
$$1=G_{0} \unlhd G_{1}\unlhd \cdots
\unlhd G_{n}\unlhd G_{n+1}=G.$$ Let  $k$ be the number of factors
which are $p$-groups(non-trivial) in $G_{i}/G_{i-1}, i=1,2,...,n+1$.

\begin{theorem} With the same notation as above, if $E_{j}$ is a
simple module of $G/G_{n}$, then the Cartan number $c_{jj}\geq k+1$,
where $c_{jj}$ is the Cartan number of $G$ corresponding to $E_{j}$.
In particular, if $G/G_{n}$ is a $p$-group, then $c_{11}\geq
|G/G_{n}|+k-1$.
\end{theorem}
\begin{proof} First, we should know $E_{j}$ is also a simple
$G/G_{i}$-module for $i=0,1,..., n-1$, as $G_{i}$ acts trivially on
$E_{j}$.

When $i=1$, we have $c_{jj}=c_{jj}^{(1)}+x_{jj}^{(1)}$ by Theorem
4.3, where $c_{jj}^{(1)}$ is the Cartan number of $F[G/G_{1}]$
corresponding to $E_{j}$ and $x_{jj}^{(1)}$ is the Cartan number of
$F[G]/I$ corresponding to $E_{j}, I=\{\alpha \in
F[G]|J(F[G_{1}])\alpha =0\}$.

Considering $c_{jj}^{(1)}$ in $F[G/G_{1}]$ and $G_{2}/G_{1}$ as
normal subgroup of $G/G_{1}$, we have
$c_{jj}^{(1)}=c_{jj}^{(2)}+x_{jj}^{(2)}$, where $c_{jj}^{(2)}$ is
the Cartan number of $F[G/G_{2}]$ corresponding to $E_{j}$ and
$x_{jj}^{(2)}$ is the Cartan number of $F[G/G_{1}]/I_{1}, I_{1}={\rm
Ann}_{r}J(F[G_{2}/G_{1}])$.

We continue in the same way as above, then we end in $n$th step:
$$c_{jj}=c_{jj}^{(n)}+x_{jj}^{(n)}+\cdots+x_{jj}^{(2)}+x_{jj}^{(1)},$$
where $c_{jj}^{(n)}$ is the Cartan number of $G/G_{n}$ corresponding
to $E_{j}$. Notice $x_{jj}^{(t)}\neq 0$ if $G_{t}/G_{t-1}$ is a
$p$-group. If $G/G_{n}$ is a $p$-group, $c_{jj}^{(n)}\geq |G/G_{n}|$
by Corollary 4.6. This completes our proof.
\end{proof}

There is a block version to state Theorem 5.7. With the same series
of  normal subgroups of $G$ as above, for a simple module $E$ in
$G/G_{n}$, we suppose $B^{(i)}$ is the  $p$-block containing $E$ in
$G/G_{i}$. Assume $p$-block $B$ of $G$ contains the simple module
$E$, then there are series of natural epimorphisms as follows:
$$B\rightarrow^{f_{1}}B^{(1)}\rightarrow^{f_{2}}B^{(2)}\rightarrow\cdots\rightarrow^{f_{n}}B^{(n)}$$
where $f_{i}$ is the composition of natural map from $G/G_{i-1}$ to
$G/G_{i}$ and the projective to $B^{(i)}$ from the image of
$B^{(i-1)}$ in $G/G_{i}$. By Corollary 4.4, we know $B^{(i)}$ and
$B^{(i+1)}$(Assume $B^{(0)}=B,i=0,1,...,n-1$) have the same Cartan
matrix if and only if $B^{(i)}$ covers a zero-defect $p$-block in
$F[G_{i+1}/G_{i}]$. Let $k$ be the number of $p$-blocks $B^{(i)}$
which covers non-zero-defect $p$-blocks, $i=0,1,...,n-1$, then we
have
\begin{theorem}Under the same notation as above, let $c$ be the Cartan
number corresponding to simple module $E$ in $G$. If $E$ belongs to
a $p$-block $B^{(n)}$, then $c\geq k+1$.
\end{theorem}
\begin{proof} By the arguments above.
\end{proof}

\section{Cartan invariants and dimensions of Brauer characters }

For $N \unlhd G$ and  a $p$-block $b$ in $N$, we let $
\textrm{Bl}(G|b)$ denote all blocks of $G$ covering $b$.  it is an
important to find connections between them. In general case, it is
not easy. If there exists defect zero blocks in $N$, we will have
more advantages to study those blocks in $G$ covering defect zero
blocks in $N$. In this section we consider some $p$-blocks of $G$
which cover a $p$-block $b$ with only one Brauer irreducible
character in $N$ for $N\unlhd G$. We obtain some results about
Cartan invariants and degrees of Brauer characters  of $B$. In
\cite{willems} and \cite{holm}, T. Holms and W. Willems present
several conjectures on Brauer character degrees. Some true cases are
proved by them. Here we show some more general true cases covering
their special cases.¡¡Furthermore in some cases, we give affirmative
answers to a question presented by T. Holms and W. Willems in
\cite{holm}.

If $E_{i}$ denote a simple $G$-module covering a simple $N$-module
$V$, by Clifford Theorem we have as $N$-module:
$$ E_{i}=e_{i}\bigoplus_{t\in G/T_{V}}V^{t},$$
where $e_{i}$ is the ramification coefficient of $E_{i}$, $T_{V}$ is
the initial group of $V$, and $V^{t}$ is a conjugate module of $V$.

Now let $P_{i}$ is the projective cover of $E_{i}$. As $N$-module,
$P_{i}$ or its character form has a similar decomposition. See W.
Feit \cite[Lemma 1.3, Ch. VI]{feit}, H. Nagao and Y.
Tsushima\cite[P389, Problem 10]{nag}, and Navarro\cite[Corollary
8.8, VIII]{nav}. Here we will improve their result with a new proof.
Let $P_{V}$ denote the projective cover of $V\in \textrm{IBr}(N)$,
then $P_{V}^{t}$ is a conjugate module of $P_{V}$ for $t\in G$ and
projective cover of $V^{t}$.

\begin{theorem}With the same notation as above, if
$$ E_{i}=e_{i}\bigoplus_{t\in G/T_{V}}V^{t},$$ then
$$P_{i}=a_{i}\bigoplus_{t\in G/T_{V}}P_{V}^{t},$$ where $a_{i}\geq e_{i}$,
the equality holds if and only if $E_{i}$ is $N$-projective.

\end{theorem}

\begin{proof}
By Frobenius-Nakayama-reciprocity theorem, there is $$
P_{V}^{G}=\bigoplus _{i}e_{i}P_{i}.$$so
$(P_{i})_{N}|(P_{V}^{G})_{N}= \bigoplus_{t\in G/N}P_{V}^{t}$. Since
all $P_{V}^{t},t\in G$ are projective indecomposable (as projective
cover of $V^{t}$), so we have $(P_{i})_{N}=\bigoplus_{t}P_{V}^{t}$
and $\textrm{Hd}((P_{i})_{N})=\bigoplus_{t}\textrm{Hd}(P_{V}^{t})$ .

Now as $G$-module, it holds $P_{i}J(F[N])\leq P_{i}J(F[G])$, so we
have exact series:
$$0\longrightarrow P_{i}J(F[G])/P_{i}J(F[N]) \longrightarrow
P_{i}/P_{i}J(F[N])\longrightarrow P_{i}/P_{i}J(F[G])\longrightarrow
0.$$

As $N$-modules, the exact series above is split since
$P_{i}/P_{i}J(F[N])$ is completely reducible. Thus
$$(\textrm{Hd}(P_{i}) )_{N}=(P_{i}/P_{i}J(F[G]))_{N}|(P_{i}/P_{i}J(F[N]))_{N}= \textrm{Hd}((P_{i})_{N}).$$
As $E_{i}=\textrm{Hd}(P_{i})$, It follows that
$$(P_{i})_{N}=\bigoplus_{t\in G/T_{V}}a_{t}P_{V}^{t}, a_{t}\geq e_{i}.$$

Since
$$\begin{array}{lll} a_{t} & =
&\textrm{Dim}_{F}((P_{i})_{N},V^{t})^{N}\\
& = & \textrm{Dim}_{F}(P_{i},(V^{t})^{G})^{G}\\
& = & \textrm{Dim}_{F}(P_{i},V^{G})^{G}\\
& = &  \textrm{Dim}_{F}((P_{i})_{N},V)^{N}\\
& = & a_{1} \end{array}$$ so let $a_{i}=a_{t}$ which only depends
$P_{i}$. Comparing $e_{i}$ and $a_{i}$, the arguments above tell us
that  $a_{i}=e_{i}$ if and only if $P_{i}J(F[G])=P_{i}J(F[N])$,
which is equivalent to that $E_{i}$ is $N$-projective\cite[Lemma
2.1, Ch. VI]{feit}. Our proof is finished.
\end{proof}

In order to understand Willems' conjecture, let
$\{\varphi_{i}|i=1,2,...,l(B)\}=\textrm{IBr}(B)$ and $\Phi_{i},
i=1,2,...,l(B)$ is the Brauer character of principal indecomposable
module corresponding to $\varphi_{i}$. We denote defect group of
block $B$ by $D_{B}$. The number $e_{i}$ and $a_{i}$, appeared in
Theorem 6.1, are the ramification coefficient of $\varphi_{i}$ and
$\Phi_{i}$, respectively.

\begin{corollary}With the same notations as above, Let $N$ be a normal subgroup of $G$ and  $B\in
\textrm{Bl}(G|b)$ for $b\in \textrm{Bl}(N)$. Let $l(b)=1$, then
$$\frac{\textrm{Dim}_{F}B }{|D_{b}|}\geq
\sum_{i=1}^{l(B)}\varphi_{i}(1)^{2},$$ the equality holds if and
only if $D_{B}\leq N$. In particular, if $B$ has a normal defect
subgroup of $G$, then
$$\frac{\textrm{Dim}_{F}B }{|D_{B}|}=
\sum_{i=1}^{l(B)}\varphi_{i}(1)^{2},$$
\end{corollary}
\begin{proof}By our condition, we suppose that
$\textrm{IBr}(b)=\theta$ and $\Psi$ is the only Brauer character
offered by  the principal indecomposable module of $b$. Since
$$\begin{array}{lll}
\textrm{Dim}_{F}(B) & = &
\sum_{i=1}^{l(B)}\varphi_{i}(1)\Phi_{i}(1)\\
& = & \sum_{i=1}^{l(B)}\varphi_{i}(1)a_{i}|G/T_{\theta}|\Psi(1)\\
& = &
\sum_{i=1}^{l(B)}\varphi_{i}(1)a_{i}|G/T_{\theta}||D_{b}|\theta(1)\\
& \geq & \sum_{i=1}^{l(B)}\varphi_{i}(1)^{2}|D_{b}|.

\end{array}$$
The equality above holds if and only if all $a_{i}=e_{i}, i=1, 2,
....., l(B)$, which is equivalent to $D_{B}\leq N$ by \cite[Theorem
2.3, Ch. VI]{feit}, or $D_{B}=D_{b}$.
\end{proof}

By Corollary 6.2, if $B$ covers a block $b$ of $N$ with $l(b)=1$ and
$D_{B}\leq N$, then Willems' local conjecture holds. In particular,
his conjecture holds for a block $B$ with a normal defect subgroup.

Without the condition $l(b)=1$, the following result present another
true case for Willems' local conjecture. In fact, Willems'
conjecture holds for $p$-solvable groups by Holm and
Willems\cite{holm}. We prove the conjecture in a weaker condition:
there exists a normal and $p$-solvable subgroup $N$ containing the
defect group $D_{B}$ of block $B$. We suppose that $\varphi_{i}$ lie
over $ \theta_{i}\in \textrm{IBr}(b)$ and that $\Psi_{i}$ denotes
the projective Brauer character¡¡ corresponding¡¡ to $\theta_{i}$.

\begin{corollary} With the same notations as above, for a normal
subgroup $N$ and $B\in \textrm{Bl}(G|b), b\in \textrm{Bl}(N)$, if
defect group $D_{B}\leq N$ and $N$ is $p$-solvable group, then
$$\frac{\textrm{Dim}_{F}(B)}{|D_{B}|}\leq
\sum_{i=1}^{l(B)}\varphi_{i}(1)^{2},
$$ so Willems' local
conjecture holds for these blocks $B$.
\begin{proof} First notice $D_{b}=D_{B}$ from $D_{B}\leq N$.
By our condition, all $e_{i}=a_{i},i=1, 2, ..., l(B)$.
Hence we have
$$\begin{array}{lll}
\textrm{Dim}_{F}(B) & = &
\sum_{i=1}^{l(B)}\varphi_{i}(1)\Phi_{i}(1)\\
& = & \sum_{i=1}^{l(B)}\varphi_{i}(1)
a_{i}|G/T_{\theta_{i}}|\Psi_{i}(1)\\
& = & \sum_{i=1}^{l(B)}\varphi_{i}(1)
a_{i}|G/T_{\theta_{i}}||N|_{p}\theta_{i}(1)_{p'}\\
&  &  ( \textrm{By \cite[ Theorem 3.2, Ch. X]{feit}})\\
& \leq & \sum_{i=1}^{l(B)}\varphi_{i}(1)
a_{i}|G/T_{\theta_{i}}||D_{b}|\theta_{i}(1) \\
&  &  (\textrm{as } |N|_{p}\leq|D_{b}|\theta_{i}(1)_{p})\\
& = & \sum_{i=1}^{l(B)}\varphi_{i}(1)^{2}|D_{B}|

\end{array}$$which induces our result.
\end{proof}

\end{corollary}
Let $C$ denote the Cartan matrix of a $p$- block $B$, with defect
group $D$. There is a question from T. Holm and W.
Willems\cite{holm}: $$ \textrm{Tr}(C)\leq l(B)|D|?$$

The question is important as it implies Willems' conjecture. We give
affirmative answer in some cases for their question. we keep to use
the notations as before.

\begin{lemma}
For a normal subgroup $N$ of  $G$ and $b\in {\rm Bl}(N)$, suppose
$B\in {\rm Bl}(G|b)$, with Cartan matrix $C=(c_{ij})$. Let $D_{B}$
and $D_{b}$ denote the defect groups of $B$ and $b$, respectively.
If $l(b)=1$, then
\begin{enumerate}
 \item
$$c_{ii}\leq \frac{a_{i}}{e_{i}}|D_{b}|,
c_{ij}\leq\frac{a_{i}}{e_{j}}|D_{b}|.$$
\item Furthermore if the simple module $E_{i}$ affording $\varphi_{i}$
is $N$-projective, then
$$c_{ii}\leq |D_{b}|.$$

\end{enumerate}
\end{lemma}
\begin{proof} Let $\textrm{IBr}(b)=\{\theta\}$ and $\Psi$ denote the projective
 Brauer character corresponding to $\theta$. Since
$$ \begin{array}{lll} \Phi_{i}(1) & = &
\sum_{i=1}^{l(B)}c_{ij}\varphi_{j}(1)\\
& = & \sum_{i=1}^{l(B)}c_{ij}e_{j}|G/T_{\theta}|\theta(1)\\
& = & a_{i}|G/T_{\theta}|\Psi(1)\\
& = & a_{i}|G/T_{\theta}||D_{b}|\theta(1),
\end{array}$$
then we get $$|D_{b}|a_{i}=\sum_{j=1}^{l(B)}c_{ij}e_{j}\geq
c_{ij}e_{j}.$$ which induces (1).

When $E_{i}$ is $N$-projective, then $a_{i}=e_{i}$. So we get (2).
\end{proof}

The following result gives affirmative answers to Wellems' question
in some cases.
\begin{theorem}For $N\unlhd G$, suppose $p$-block $B$ has a defect group $D_{B}$ and covers a
$p$-block $b$ in $N$. If $l(b)=1$ and $D_{B}\leq N$, then we have
$$ {\rm Tr}(C) \leq l(B)|D_{B}|.$$In particular, if $B$ has a
normal defect group $D_{B}$, then $${\rm Tr}(C)\leq l(B)|D_{B}|.$$
\end{theorem}
\begin{proof} By our condition, all $a_{i}=e_{i}, i=1, 2, ..., l(B)$ and
$D_{B}=D_{b}$, defect group of $b$, so our result follows from Lemma
6.4.

\end{proof}
Since T. Holm and W. Willems\cite{holm} has proved that
$$\frac{\textrm{Dim}_{F}(B)}{\textrm{Tr}(C)}\leq\sum_{i=1}^{l(B)}\varphi_{i}(1).$$
Thus by Theorem 6.5 we show again that their conjecture holds for a
block $B$, if it covers a block $b$ in $N$, with $l(b)=1$ and
$D_{B}\leq N$, or if $B$ has a normal defect group,  as proved by
Corollary 6.2.

\section{Some applications for our results}

In this section, we further consider applications for our results in
the Sections before, Specially when $N$ is a normal $p$-subgroup of
$G$. For example, Brauer Problem VII \cite[Ch. IV, \S5]{feit}is
about bounds of Cartan invariants in terms of defect groups. If $G$
is $p$-solvable, the upper bound is the order of defect group by
Fong \cite[Ch. X, \S4]{feit}. Under our condition, we will give some
bounds for Cartan invariants. These results can help us understand
the connection between the group invariants of $G$ and algebraic
invariants of $F[G]$. We keep the same notation as before.

The proposition 3.1 will be improved with the new condition. The
part 1 in the following result can be obtained directly from
\cite{tsu}, but we prove it here in a simple way.
\begin{proposition} Let $A=F[G]$ and  $N$ be a normal $p$-subgroup
of $G$. Then:
\begin{enumerate}
\item Let $\widehat{N}=\sum_{x\in N}x \in Z(A)$, the center of
$A$, then $$I=\{\alpha=\sum_{y\in G/N}a_{y}\widehat{N}y,a_{y}\in
F\}=A\widehat{N},$$ an ideal of $A$ generated by $\widehat{N}$.
   \item  $I^{2}=0$, and {\rm Soc}($G)\leq I\leq J(A)$, where {\rm Soc}$(G)$ is the socle of
   $A$
   \item If $A=\bigoplus_{i}P_{i}$,  a decomposition of principal indecomposable module of $A$,
   then $A/I=\bigoplus_{i}P_{i}/P_{i}\cap
   I$, a similar decomposition of $A/I$ with the same number of summands as in $A$.

   \end{enumerate}
\end{proposition}
\begin{proof} Notice $I=r(J(F[N]))=\{\alpha \in
F[G]|J(F[N])\alpha=0\}$, and $J(F[N])=\langle x-1|g\in N \rangle$,
an ideal of $F[N]$ generated by $x-1, x\in N$. Let
$\alpha=\sum_{g\in G}a_{g}g \in I$, then for any $x\in N$,
$$\begin{array}{lll}(x-1)\alpha &= &(x-1)\sum_{g}a_{g}g\\ & =
&\sum_{g}a_{g}xg-\sum_{g}a_{g}g\\
 &= &\sum_{y}(a_{x^{-1}y}-a_{y})y\\ &= &0. \end{array}$$
Thus $a_{x^{-1}y}=a_{y}$, for any $x\in N, y\in G$. Assertion 1
follows.

Notice $\widehat{N}^{2}=|N|\widehat{N}=0,$ Hence $I^{2}=0 , I\leq
J(A)$ by assertion 1. Since $J(F[N])$Soc($G)\leq J(A)$Soc$(G)=0$,
Soc$(G)\leq I$. Assertions 2 follows.

Since we have $J(A/I)=J(A)/I$ by Proposition 3.1 and assertion 2,
thus
$${\rm Hd}( A/I)={\rm Hd}(A)=A/J(A),$$ so there is no $P_{i}$ such that
$P_{i}\cap I=P_{i}$. Assertion 3 follows.
\end{proof}

On the bounds of Cartan invariants, we have following results:
\begin{corollary} If simple module $E_{i}$ has a normal vertex $N$, then
$$c_{ii}\leq |N|$$
\end{corollary}
\begin{proof} Easy to see by Lemma 6.4.
\end{proof}

\begin{theorem}For $N\lhd G$ and $b\in {\rm Bl}(N)$, suppose
$B\in {\rm Bl}(G|b)$, with a defect group $D_{B}$. If $l(b)=1$ and
$D_{B}\leq N$, then$$c_{ij}\leq |D_{B}|, i,j=1,2,...,l(B).$$In
particular, if $D_{B}$ is a normal defect group of $G$, then
$$c_{ij}\leq |D_{B}|, i,j=1,2,...,l(B).$$
\end{theorem}
\begin{proof}According to Lemma 6.4, we have $$c_{ii}\leq
|D_{B}|,i=1,2,...,l(B).$$ Hence
$$c_{ij}\leq\sqrt{c_{ii}c_{jj}}\leq|D_{B}|,$$ for any $1\leq i,j\leq
l(B)$.
\end{proof}

Applying Theorem 5.7 and Corollary 7.2, we give a result about a
series of normal subgroups of $G$.
\begin{corollary}Suppose we have a series of normal subgroup of $G$:
$$1=G_{0} \unlhd G_{1}\unlhd \cdots
\unlhd G_{n}\unlhd G_{n+1}=G.$$ Let  $k$ be the number of factors
which are $p$-groups(non-trivial) in $G_{i}/G_{i-1}, i=1,2,...,n+1$.
If $E_{i}$ is a simple $G/G_{n}$-module and has a normal vertex $N$,
then $k+1\leq c_{ii}\leq |N|$.
\end{corollary}
\begin{proof} Easy to see by Theorem 5.7 and Corollary 7.2.

\end{proof}

We obtain a result about eigenvalues of Cartan matrix $C=(c_{ij})$
of $B$, which generalize a result when $D_{B}$ is  a normal defect
group. \cite{wada}\cite[Ch. IV, \S4.26]{feit}.
\begin{theorem} For $N\lhd G$ and $b\in {\rm Bl}(N)$, suppose
$B\in {\rm Bl}(G|b)$, with defect group $D_{B}$. If $l(b)=1$ and
$D_{B}\leq N$, then $|D_{B}|$ is an eigenvalue of Cartan matrix
$C=(c_{ij})$ of $B$ with an eigenvector
$$(e_{1}, e_{2}, ..., e_{l(B)}).$$In particular, the conclusion
holds for a $p$-block $B$ with a normal defect group and the
eigenvector is
$$(\varphi_{1}(1), \varphi_{2}(1), ..., \varphi_{l(B)}(1)).$$
\end{theorem}
\begin{proof} As proof in Lemma 6.4, we get $$|D_{b}|a_{i}=\sum_{j=1}^{l(B)}c_{ij}e_{j}.$$
 Notice all $a_{i}=e_{i}, i=1,2, ..., l(B)$, and $D_{b}=D_{B}$ under our condition.
So $$
\left(\begin{array}{l}e_{1}\\e_{2}\\\vdots\\e_{l(B)}\end{array}\right)|D_{B}|
=
(c_{ij})\left(\begin{array}{l}e_{1}\\e_{2}\\\vdots\\e_{l(B)}\end{array}\right)$$
by which the assertion follows. When $N=D_{B}\lhd G$, all
$e_{i}=\varphi_{i}(1),i=1,2,..., l(B)$.

\end{proof}

\subsection*{Acknowledgement} The author was kindly supported by GZ 301 of Sino-Germany Academic Exchange Center
visiting Jena University for one month. Visiting and talking with
Professor B. Kulshammer helped the author greatly  in writing this
paper, so the author thanks Professor B. Kulshammer and Jena
University for their hospitality.
\bibliographystyle{amsplain}

\begin{thebibliography}{99}
\bibitem{acs}J. L. Alperin, M. Collins and D. Sibley, \textit{Projective
modules, filtrations, and Cartan invariants,} Bull. L. M. S.,
Vol16(1984), 416-420

\bibitem{feit} W. Feit, \textit{The representation theory of finite
groups,}
North-Holland Publishing Company, Amsterdam, 1980

\bibitem{wada} M. Kiyota, M. Murai and T. Wada, \textit{Rationality of eigenvalues of Cartan matrices in  finite
groups}, J. Algebra, 249(2002), 110-119

\bibitem{kul} Laszlo Hethelyi, Erzsebet Horvath, Burkhard
Kulshammer, John Murray, \textit{Central ideals and Cartan
invariants of symmetric algebras, }J. Algebra, 296(2006), 177-195

\bibitem{kuls} Burkhard Kulshammer, \textit{Group-theoretical descriptions
of ring-theoretical invariants of group algebras,} Progress in
mathematics, Vol.95(1991), 425-442

\bibitem{holm} Thorsten Holm and Wolfgang Willems, \textit{A local
conjecture on Brauer character degrees of finite groups,} Trans. Am.
Math. Soc., 359(2007), 591-603

\bibitem{rob}R. Knorr and G. Robinson, \textit{Some remarks on a conjecture of Alperin,
}J. London Math. Soc. (2)39(1989), 48-60

\bibitem{lan} P. Landrock, \textit{Finite group algebras and their
modules,} Cambridge Univ. Press, Cambridge, 1983

\bibitem{nag} Hirosi Nagao and Yukio Tsushima, \textit{Representations of
finite groups,} Academic Press, INC., 1989

\bibitem{nav} G. Navarro, \textit{Characters and blocks of finite groups,}
 L. M. S. Lecture Note Series 250, Cambridge
University Press, 1998

\bibitem{shi} Shigeo Koshitani, \textit{Cartan invariants of group algebra of finite
groups,} Proceedings of the AMS, 124(8)(1996), 2319-2323

\bibitem{tsus} Y. Tsushima, \textit{On the block of defect
zero,} Nagoya Math. J. 44(1971), 57-59

\bibitem{tsu} Y. Tsushima, \textit{On the annihilator ideals of the
radical of a group algebra,} Osaka J. Math. 8(1971), 91-97

\bibitem{willem}  W. Von Willems, \textit{On the projectives of a group algebra,} Math.
Z.171(1980), 163-174

\bibitem{willems} W. Willems, \textit{On degrees of irreducible Brauer
characters}, Trans. Am. Math. Soc., Vol. 357(2005), 2379-2387

\bibitem{some} Zeng Jiwen, \textit{Some results on the Cartan matrix of a Frobenius
Algebra,} Communications in algebra, 24(14)(1996), 4385-4396


\end{thebibliography}

\end{document}